\documentclass[12pt]{article}
\usepackage[centertags]{amsmath}
\usepackage{amsfonts}
\usepackage{amssymb}
\usepackage{amsthm}  
\usepackage{comment}
\usepackage{natbib}
\input{colordvi}  
\usepackage{epsfig}         
\usepackage{graphicx}
\usepackage{bm}    
\usepackage{cancel}  %
\usepackage{mathrsfs} 
\newlength{\defbaselineskip}
\newcommand{\setlinespacing}[1]%
           {\setlength{\baselineskip}{#1 \defbaselineskip}}

\def\o*{o_{{\!}_{P^*}}}
\def\O*{\cal O_{{\!}_{P^*}}}

\def\Var{\mathrm{var}}

\def\.{\mbox{.}}

\def\le{\leqslant}
\def\ge{\geqslant}
\def\sfrac(#1,#2){\mbox{$\frac{#1}{#2}$}}
\def\comb(#1,#2){\mbox{${#1\choose #2}$}}

\def\tr{{\mbox{\tiny{$\mathrm{T}$}}}}

\def\bernstein{Bernstein }

\def\Err{\mathrm{Err}}
\def\lfhook#1{\setbox0=\hbox{#1}{\ooalign{\hidewidth
    \lower1.5ex\hbox{'}\hidewidth\crcr\unhbox0}}}
\newtheorem{theorem}{Theorem}
\newtheorem{corollary}[theorem]{Corollary}
\newtheorem{lemma}[theorem]{Lemma}

\newtheorem{rem}{Remark}[section]
\newtheorem{example}{\bf Example}

\title{Iterated \bernstein polynomial approximations}
\author{Zhong Guan\\
Department of Mathematical Sciences, \\
Indiana University South Bend,
\\
1700 Mishawaka Avenue, P.O. Box 7111 \\
South Bend, IN 46634-7111,
U.S.A.\\
zguan@iusb.edu}

\date{}

\begin{document}


\maketitle
\begin{abstract}
Iterated \bernstein polynomial approximations of degree $n$ for continuous function which also use the values of the function at $i/n$, $i=0,1,\ldots,n$, are proposed.
The rate of convergence of the classic \bernstein polynomial approximations is significantly improved
by the iterated \bernstein polynomial approximations without increasing the degree of the polynomials.
The close form expression of the limiting iterated \bernstein polynomial approximation  of degree $n$
when the number of the iterations approaches infinity is obtained.
The same idea applies to the $q$-\bernstein polynomials and the Szasz-Mirakyan approximation.
The application to numerical integral approximations which gives surprisingly good results is also discussed.\\
\\
\vspace{1em}
\end{abstract}

\section{Introduction}
The \bernstein polynomials \citep{Bernstein} have been used for
approximations of functions in many areas of mathematics and other
fields such as smoothing in statistics and constructing B\'{e}zier
curves \cite[see][for examples]{Gal-S-G-book-2008,
Pena-J-M-book-1999} which have important applications in computer
graphics. One of the advantages of the \bernstein polynomial
approximation of a continuous function $f$  is that it approximates
$f$  on $[0, 1]$ uniformly using only the values of $f$ at $i/n$,
$i=0, 1, \ldots, n$. In case when the evaluation of $f$ is difficult
and expensive, the \bernstein polynomial approximation is preferred.

The properties of the \bernstein polynomial approximation have been
studied extensively by many authors for decades.
 However the slow optimal rate $\mathcal{O}(1/n)$ of convergence of the classical \bernstein
 polynomial approximation  makes it not so attractive. Many authors have made tremendous
 efforts to improve the performance of the classical \bernstein polynomial approximation.
 Among many others,  Butzer\cite{Butzer-1953} introduces linear combinations of the
 \bernstein polynomials  and  Phillips\cite{Phillips-G-M-1997-q-Bernstein} proposes
 the $q$-\bernstein polynomials which is a generalization of the classical \bernstein
 polynomial approximation. However, Butzer\cite{Butzer-1953}'s approximation involves
 not only the \bernstein polynomials of degree $n$ but also degree of $2n$ which
 requires more sampled values of the function to be approximated at the $2n+1$
 rather than $n+1$ uniform partition points of $[0, 1]$. The $q$-\bernstein polynomial
 approximates a function $f$ only when $q\ge1$. For $q>1$, it seems that $f(z)$ has to
 be an analytic complex function on disk $\{z\,:\, |z|<r\}$, $r>q$, so that the $q$-\bernstein
 polynomial approximation  of degree $n$ has a better rate of convergence, $\mathcal{O}(q^{-n})$,
 than the best rate of convergence, $\mathcal{O}(n^{-1})$, of  the classical \bernstein polynomial approximation of degree $n$ \citep[see][for example]{Ostrovska-2003-JAT,Wang-Wu-2008-JMAA-Saturate-Conv-q-Bernstein}. If $q>1$, the $q$-\bernstein polynomial approximation of degree $n$ uses the sampled values of the function at $n+1$ nonuniform
 partition points of $[0,1]$. These points except $t=1$ are attracted toward $t=0$ when $q$ is getting larger so that the approximation becomes worse in the the neighborhood of the right end-point.  This is a serious drawback of the $q$-\bernstein polynomial approximation which limits the scope of its applications.

 In this paper, we propose a simple procedure to generalize and improve the classical \bernstein polynomial approximation by repeatedly approximating the errors using the \bernstein polynomial approximations. This method involves only the iterates of the \bernstein operator applied on the base \bernstein polynomials of degree $n$ and the sampled values of the function being approximated at the same set of $n+1$ uniform partition points of $[0, 1]$. The improvement made by the $q$-\bernstein polynomial approximation with properly chosen $q$ can be achieved by the iterated \bernstein polynomials without messing up the right boundary.

 \section{Preliminary Results About the Classical \bernstein Polynomial}

 Let $f$ be a function on $[0, 1]$. The classical \bernstein polynomial of degree $n$ is defined as
\begin{equation}\label{n-th order bernstein approximation}
    \mathbb{B}_{n}f(t)=\mathbb{B}_{n}^{(1)}f(t)=\sum_{i=0}^{n}
f\big(\sfrac({i},{n})\big)B_{ni}(t),\quad 0\le t\le 1,
\end{equation}
where $\mathbb{B}_{n}$ is called the \bernstein operator and $B_{ni}(t)={n\choose i}t^i(1-t)^{n-i}$, $i=0,\ldots,n$,
are called the \bernstein basis polynomials.  Note that the \bernstein polynomial of degree $n$, $\mathbb{B}_{n}^{(1)}f$, uses only the sampled values of $f$ at $t_{ni}=i/n$, $i=0, 1, \ldots, n$.
Note also that for $i=0,\ldots,n$, $$\beta_{ni}(t)\equiv (n+1)B_{ni}(t),\quad 0\le t\le 1,$$ is the density function of beta distribution  $beta(i+1, n+1-i)$. Let $Y_n(t)$ be a binomial $b(n, t)$ random variable. Then $E\{Y_n(t)\}=nt$, $\Var\{Y_n(t)\}=E\{Y_n(t)-nt\}^2=nt(1-t)$,  $E\{Y_n(t)-nt\}^3=nt(1-t)(1-2t)$, and   $\mathbb{B}_{n}f(t) =E[f\{Y_n(t)/n\}]$.
The error of $\mathbb{B}_{n}^{(1)}f$ is
\begin{equation}\label{Error of the n-th order bernstein approximation}
    \Err\{\mathbb{B}_{n}^{(1)}f\}(t)=\mathbb{B}_{n}^{(1)}f(t)-f(t).
\end{equation}
Let $f$ be a member of $C^{(r)}[0,1]$, the set of all continuous functions that have continuous first $r$ derivatives. $C[0, 1]=C^{(0)}[0,1]$.  Let the modulus of continuity of the $r$th derivative $f^{(r)}$ be
$$\omega_r(\delta)=\max_{|s-t|<\delta}|f^{(r)}(s)-f^{(r)}(t)|,\quad \delta>0.$$
About the rate of convergence of $\mathbb{B}_{n}^{(1)}f$ we have the following well known results
\citep[see][]{Lorentz-1986-book-bernstein-poly}.
\begin{theorem} \label{rate of conergence of berinstein polynomial approx} Suppose $f\in C^{(r)}[0,1]$, $r=0,1$.
For each $n>1$
$$|\Err\{\mathbb{B}_{n}^{(1)}f\}(t)|=|\mathbb{B}_{n}f(t)-f(t)|\le C_r n^{-r/2}\omega_r(n^{-1/2}),$$
where $C_r$ is a constant depending on $r$ only. One can choose $C_0=5/4$ and $C_1=3/4$.
\end{theorem}
The result according to $r=0$ is due to
Popoviciu\cite{Popoviciu-1935}. The order of approximation of $f\in
C^{(r)}[0,1]$ by arbitrary polynomials is given by the theorem of
Dunham Jackson \citep{Dunham-Jackson-1930}
\begin{theorem} [Dunham Jackson] \label{thm of Dunham Jackson} Suppose $f\in C^{(r)}[0,1]$, $r\ge0$.
For each $n>r$ there exists a polynomial $P_n$ of degree at most $n$ so that
$$|P_n(t)-f(t)|\le C_r' n^{-r}\omega_r(n^{-1}),$$
where $C_r'$ is a constant depending on $r$ only. If $r=0$, one can choose $C_0'=3$.
\end{theorem}
The following is a result of Voronovskaya \cite{Voronovskaya-1932} about the asymptotic formula of the \bernstein polynomial approximation.
\begin{theorem} [E. Voronovskaya] \label{thm of Voronovskaya}  Suppose that $f$ has second derivative $f''$. Then
\begin{equation}\label{asymptotic formula of Bernstein poly}
\Err\{\mathbb{B}_{n}^{(1)}f\}(t) = \mathbb{B}_{n}f(t)-f(t)=\frac{t(1-t)}{2n}f''(t)+\frac{1}{n}\varepsilon_n(t),
\end{equation}
where $\varepsilon_n(t)$ is a sequence of functions which converge to 0 as $n\to\infty$.
\end{theorem}
From Theorem \ref{thm of Voronovskaya} it follows that the best rate of convergence of $\mathbb{B}_{n}^{(1)}f$, as $n\to \infty$,
is $\mathcal{O}(n^{-1})$ even if $f$ has continuous second or higher order derivatives \citep{Lorentz-1986-book-bernstein-poly}. This is not as good as in the case of arbitrary polynomial approximation in which if $f$ has continuous $r$th   derivative  then the rate of convergence of a sequence of arbitrary polynomials $P_n$ of degree at most $n$  can be at least $o(n^{-r})$ \citep{Dunham-Jackson-1930}.
\bernstein \cite{Bernstein-1932} generalizes this asymptotic formula to contain terms up to the $(2k)$th derivative and proposes a polynomial constructed based on both $f(i/n)$ and $f''(i/n)$, $i=0,1,\ldots,n$.
Butzer \cite{Butzer-1953} considers some 
combinations of \bernstein polynomials of different degrees
and shows that  they 
have better rate of convergence
which is much faster than $\mathcal{O}(1/n)$.  Costabile et al
\cite{Costabile-etal-BIT-1996} generalize the linear combinations of
the
 \bernstein polynomials proposed by of \cite{Butzer-1953}, \cite{Frentiu-Lin-Comb-Bernstein-poly-1970}  and \cite{May-Canad-J-Math-1976}. The $q$-\bernstein polynomials of
\cite{Phillips-G-M-1997-q-Bernstein} has better rate of convergence. However, if $0<q<1$, the $q$-\bernstein polynomials 
of function $f$ do not approximate $f$. For $q>1$, 
the $q$-\bernstein polynomials do approximate $f$  at a rate of $\mathcal{O}(q^{-m})$ but $f(z)$ has to be analytic in a complex disk with radius greater than $q$.
The analyticity of $f$ may be too restrictive for applications. Even if we are
sure that $f$ is analytic, we have to deal with the choice of $q$. In some cases, the approximations are very sensitive to the choice of $q$.

\section{The Iterated \bernstein Polynomials and the Rate of Convergence}\label{sect: iterated Bernstein poly and rate of convergence}
The error $\Err\{\mathbb{B}_{n}^{(1)}f\}(t)$ is also a continuous
function on $[0,1]$ whose values  at $t_i=i/n$, $i=0, 1, \ldots, n$,
depend on  $f(t_i)$, $i=0, 1, \ldots, n$, only. So we can
approximate this error function by the \bernstein polynomial
$\mathbb{B}_{n}^{(1)}[\Err\{\mathbb{B}_{n}^{(1)}f\}](t)$ and then
subtract the approximated error function from
$\mathbb{B}_{n}^{(1)}f(t)$ to obtain the second order \bernstein
polynomial of degree $n$
\begin{equation}\label{The 2nd-n-th order bernstein approximation}
    \mathbb{B}_{n}^{(2)}f(t)=\mathbb{B}_{n}^{(1)}f(t)
    -\mathbb{B}_{n}^{(1)}[\Err\{\mathbb{B}_{n}^{(1)}f\}](t).
\end{equation}
This idea is closely related to, although was not
initiated by, the proposal of \bernstein \cite{Bernstein-1932} in which the second derivative rather than the error of the \bernstein polynomial is approximated.
Inductively,
\begin{equation}\label{eq: inductive formula of iterated bernstein polynomial}
    \mathbb{B}_{n}^{(k+1)}f(t)=\mathbb{B}_{n}^{(k)}f(t)-\mathbb{B}_{n}\{\mathbb{B}_{n}^{(k)}f(t)-f(t)\},\quad k\ge 1.
\end{equation}
This iteration procedure can be performed further until a satisfactory approximation precision is achieved because the error
$\Err\{\mathbb{B}_{n}^{(k)}f(t)\}=\mathbb{B}_{n}^{(k)}f(t)-f(t)$ can be estimated by $\mathbb{B}_{n}\{\mathbb{B}_{n}^{(k)}f(t)-f(t)\}
=\mathbb{B}_{n}^{(k)}f(t)-\mathbb{B}_{n}^{(k+1)}f(t)$.
\begin{lemma}
Generally the $k$-th order \bernstein polynomial of degree $n$ can be written as
\begin{equation}\label{The k-th-n-th order bernstein approximation}
    \mathbb{B}_{n}^{(k)}f(t)=\sum_{i=1}^k{k\choose i}(-1)^{i-1}\mathbb{B}_{n}^if(t),\quad k\ge 1,\quad 0\le t\le 1.
\end{equation}
Define
$\mathbb{B}_{n}^0f(t)=f(t)$.
Then the error of the $k$-th  \bernstein polynomial   of degree $n$ can be written as
\begin{equation}\label{error of k-th bernstein poly}
\Err\{\mathbb{B}_{n}^{(k)}f(t)\}=\mathbb{B}_{n}^{(k)}f(t)-f(t)
=\sum_{i=0}^k{k\choose i}(-1)^{i-1}\mathbb{B}_{n}^if(t)=-(\mathbb{I}-\mathbb{B}_{n})^kf(t),
\end{equation}
where $\mathbb{I}=\mathbb{B}_{n}^0$ is the identity operator.
\end{lemma}
\begin{proof}
\begin{align}\nonumber
    \mathbb{B}_{n}^{(k+1)}f(t)&=\mathbb{B}_{n}^{(k)}f(t)-\mathbb{B}_{n}\{\mathbb{B}_{n}^{(k)}f(t)-f(t)\}
    \\\nonumber
    &=\sum_{i=1}^k{k\choose i}(-1)^{i-1}\mathbb{B}_{n}^if(t)-\sum_{i=1}^k{k\choose i}(-1)^{i-1}\mathbb{B}_{n}^{i+1}f(t)+\mathbb{B}_{n}f(t)
    \\\nonumber
    &
    =\sum_{i=1}^k{k\choose i}(-1)^{i-1}\mathbb{B}_{n}^if(t)+\sum_{i=2}^{k+1}{k\choose i-1}(-1)^{i-1}\mathbb{B}_{n}^{i}f(t)+\mathbb{B}_{n}f(t)
    \\ \label{The (k+1)-st-n-th order bernstein approximation}
    &
    =\sum_{i=1}^{k+1}{k+1\choose i}(-1)^{i-1}\mathbb{B}_{n}^if(t).
\end{align}
By induction, (\ref{n-th order bernstein approximation}) and (\ref{The (k+1)-st-n-th order bernstein approximation}) assure that (\ref{The k-th-n-th order bernstein approximation}) is true for every positive integer $k$. Equation (\ref{error of k-th bernstein poly}) is then obvious.
\end{proof}

The limit of $\mathbb{B}_{n}^kf(t)$, as $k\to\infty$, has been given
by Kelisky and Rivlin \cite{Kelisky-Rivlin-PJM-1967}. A short and
elementary proof of \cite{Kelisky-Rivlin-PJM-1967}'s result is given
by \cite{Abel-Ivan-Amer-Math-Month-2009}. After we finished the
first version of this paper, we realized that \cite{Sahai-2004}
obtained the formula (\ref{error of k-th bernstein poly}) and
investigated the properties of
$\mathbb{B}_{n}^{(k)}f(t)$ using simulation method.
The cost of $\mathbb{B}_{n}^{(k)}f(t)$ is only some  simple
algebraic calculations in addition to the evaluation of $f$ at
$i/n$, $i=0,1,\ldots,n$.

%
%
About the iterates of the \bernstein operator we have the following result.
\begin{lemma}
For $k\ge 1$,
\begin{align} \label{k-th Bernstein poly of order n}
     \mathbb{B}_{n}^kf(t)&= \sum_{i=0}^{n}
f\big(\sfrac({i},{n})\big) \mathbb{B}_{n}^{k-1}(B_{ni})(t) ,\quad k\ge 1;\quad 0\le t\le 1,
\end{align}
where $\mathbb{B}_{n}^{0}B_{ni}(t)=B_{ni}(t)$, and
\begin{equation}\label{iterates of base bernstein ploys-k} \mathbb{B}_{n}^{k+1}B_{ni}(t)=\mathbb{B}_{n}\{\mathbb{B}_{n}^{k}B_{ni}\}(t),\quad k\ge1.
\end{equation}
When $k=1$,
\begin{align} \label{iterates of base bernstein ploys} \mathbb{B}_{n}^{1}B_{ni}(t)&=\sum_{j=0}^{n}B_{ni}\big(\sfrac({j},{n})\big)B_{nj}(t).
\end{align}

\end{lemma}
\begin{proof}
The  theorem can be easily proved by induction and the fact that the \bernstein operator is linear.
\end{proof}
By (\ref{The k-th-n-th order bernstein approximation}) and (\ref{error of k-th bernstein poly}) we have
\begin{theorem} The $k$-th \bernstein polynomial approximation can be calculated inductively as
\begin{align}
\label{The k-th-n-th order bernstein approximation-2}
    \mathbb{B}_{n}^{(k)}f(t)
   & =\sum_{i=0}^{n}
f\big(\sfrac({i},{n})\big)\sum_{j=1}^k{k\choose j}(-1)^{j-1}\mathbb{B}_{n}^{j-1}B_{ni}(t) ,\quad k\ge 1,\quad 0\le t\le 1.
\end{align}
Clearly, 
for every $k\ge 1$, $\mathbb{B}_{n}^{(k)}$ preserves linear functions. Therefore 
\begin{equation}\label{error of k-th bernstein poly-3}
\Err\{\mathbb{B}_{n}^{(k)}f(t)\}=\sum_{i=0}^{n}
\left\{f\big(\sfrac({i},{n})\big)-f(t)\right\}\sum_{j=1}^k{k\choose j}(-1)^{j-1}\mathbb{B}_{n}^{j-1}B_{ni}(t)  ,\quad k\ge 1;\, 0\le t\le 1.
\end{equation}
\end{theorem}
Expression (\ref{The k-th-n-th order bernstein approximation-2}) can easily implemented in computer languages using iterative algorithm.
Define indicator functions
\begin{equation}\label{indicator of ti}
    I_{ni}(t)=\left\{
                \begin{array}{ll}
                  1, & \hbox{$t=\frac{i}{n}$;} \\
                  0, & \hbox{$t\ne \frac{i}{n}$.}
                \end{array}
              \right.
\end{equation}
Then
 $B_{ni}(t)=\mathbb{B}_{n}I_{ni}(t)=\mathbb{B}_{n}^{(1)}I_{ni}(t),\quad i=0, 1, \ldots, n$, and, by Theorem \ref{The k-th-n-th order bernstein approximation}, (\ref{The k-th-n-th order bernstein approximation-2}) and (\ref{error of k-th bernstein poly-3}) can be simplified as
\begin{align}
\label{The k-th-n-th order bernstein approximation-3}
    \mathbb{B}_{n}^{(k)}f(t)
   & =\sum_{i=0}^{n}
f\big(\sfrac({i},{n})\big)\mathbb{B}_{n}^{(k)}I_{ni}(t) ,\quad k\ge 1,\quad 0\le t\le 1.\\
\label{error of k-th bernstein poly-4}
\Err\{\mathbb{B}_{n}^{(k)}f(t)\}&=\sum_{i=0}^{n}
\left\{f\big(\sfrac({i},{n})\big)-f(t)\right\}\mathbb{B}_{n}^{(k)}I_{ni}(t),\quad k\ge 1;\, 0\le t\le 1.
\end{align}
\begin{equation}\label{eq: iterated Bernstein base polynomial}
B_{ni}^{(k)}(t)=\mathbb{B}_{n}^{(k)}I_{ni}(t)
=\sum_{j=1}^k{k\choose j}(-1)^{j-1}\mathbb{B}_{n}^{j-1}B_{ni}(t).
\end{equation}
The following theorem shows that the iterated \bernstein polynomials
, like the classical ones, have no error at the endpoints of $[0,
1]$.
\begin{theorem}\label{thm: iBernstein poly at t=0 and t=1}
For any function $f$ defined on $[0, 1]$ and any integer $k\ge 0$,
\begin{equation}\label{eq: iBernstein poly at t=0 and t=1}
\mathbb{B}_{n}^{(k)}f(0)=f(0),\quad \mathbb{B}_{n}^{(k)}f(1)=f(1).
\end{equation}
 \end{theorem}
\begin{proof}
It is known that $\mathbb{B}_{n}^{0}B_{ni}(t)=B_{ni}(t)=I_{ni}(t)$
for $t=0, 1$, $i=0,\ldots,n$. Assume that
$\mathbb{B}_{n}^{k-1}B_{ni}(t)=I_{ni}(t)$ for $t=0, 1$,
$i=0,\ldots,n$, and some $k\ge 1$. By Theorem \ref{k-th Bernstein
poly of order n}, if $t=0, 1$,
$$\mathbb{B}_{n}^kB_{ni}(t)= \sum_{j=0}^{n}
B_{ni}\big(\sfrac({j},{n})\big) \mathbb{B}_{n}^{k-1}(B_{nj})(t) =
\sum_{j=0}^{n} B_{ni}\big(\sfrac({j},{n})\big)
I_{nj}(t)=B_{ni}(t)=I_{ni}(t).$$ So by induction, for all
nonnegative integers $k$, $t=0, 1$, and $i=0,\ldots,n$,
$\mathbb{B}_{n}^kB_{ni}(t)=I_{ni}(t)$. By (\ref{eq: iterated
Bernstein base polynomial}), we have, if $t=0,1$,
$$B_{ni}^{(k)}(t)=\mathbb{B}_{n}^{(k)}I_{ni}(t)
=\sum_{j=1}^k{k\choose j}(-1)^{j-1}I_{ni}(t)=I_{ni}(t).
$$
Thus by (\ref{The k-th-n-th order bernstein approximation-3}), if
$t=0,1$,
$$\mathbb{B}_{n}^{(k)}f(t)
 =\sum_{i=0}^{n}
f\big(\sfrac({i},{n})\big) I_{ni}(t)=f(t).$$
\end{proof}

Clearly, for each $k\ge 1$, $\mathbb{B}_{n}^{(k)}f(t)$ can be
written as
$$\mathbb{B}_{n}^{(k)}f(t)=\bm{F}_n^{(k)}
\bm{B}_n(t)$$ where $\bm{F}_n^{(k)}=(f_{ni}^{(k)})_{1\times(n+1)}$ is an $(n+1)$ row vector, and
$$\bm{B}_n(t)=\{B_{n0}(t),\ldots, B_{nn}(t)\}^\tr.$$
 If $k=1$,
 $$f_{ni}^{(1)}=f\big(\sfrac({i-1},{n})\big),\quad i=1, \ldots, n+1.$$
Define  $(n+1)\times(n+1)$ square matrix
$\bm{\mathfrak{B}}_n=\bm{B}_n(\bm{u}_n^\tr)=(b_{ij})_{(n+1)\times(n+1)}$
where
$$\bm{u}_n=\big(\sfrac({0},{n}),\sfrac({1},{n}),\ldots,\sfrac({n},{n})\big)^\tr.$$
That is
$$b_{ij}=B_{n,i-1}\big(\sfrac({j-1},{n})\big),\quad i, j=1, \ldots, n+1.$$
It is easy to see that $\bm{\mathfrak{B}}_n$ is nonsingular and have
all the eigenvalues  in $(0, 1]$ among them exactly two  are ones
which correspond to eigenvectors $\bm{u}_n$ and
$\bm{1}_{n+1}=(1,\ldots,1)^\tr\in R^{n+1}$. We have the following
theorem.
\begin{theorem}\label{thm: formula for matrix  Fn}
For $k\ge 1$,
\begin{equation}\label{eq: formula for matrix  Fn}
 \bm{F}_n^{(k)}=\sum_{i=1}^k{k\choose
i}(-1)^{i-1} \bm{F}_n^{(1)}\bm{\mathfrak{B}}_n^{i-1}=\bm{F}_n^{(1)}\bm{\mathfrak{B}}_n^{-1}\{I_{n+1}
-(I_{n+1}-\bm{\mathfrak{B}}_n)^k\},
\end{equation}
where $\bm{\mathfrak{B}}_n^0=I_{n+1}$, the $(n+1)$st order unit
matrix. If $k\ge 1$,
\begin{equation}\label{eq: iterative formula for matrix  Fn}
 \bm{F}_n^{(k+1)}=\bm{F}_n^{(k)}\{I_{n+1}-\bm{\mathfrak{B}}_n\}+\bm{F}_n^{(1)} .
\end{equation}
\end{theorem}
\begin{proof}
It is easy to show that
$$\bm{F}_n^{(2)}=2\bm{F}_n^{(1)}- \bm{F}_n^{(1)}\bm{\mathfrak{B}}_n.$$
By induction, (\ref{eq: formula for matrix  Fn}) and (\ref{eq:
iterative formula for matrix  Fn}) can be easily proved.
\end{proof}
More importantly, we have
\begin{theorem}\label{thm: formula for matrix  Fn when k=infty}
The  ``optimal'' \bernstein polynomial approximation of
degree $n$ is
\begin{equation}\label{eq: optimal \bernstein polynomial of degree n}
\mathbb{B}_{n}^{(\infty)}f(t)=\bm{F}_n^{(\infty)}
\bm{B}_n(t)  = \bm{F}_n^{(1)}\bm{\mathfrak{B}}_n^{-1} \bm{B}_n(t),
\end{equation}
where \begin{equation}\label{eq: limit of matrix  Fn}
 \bm{F}_n^{(\infty)}=\lim_{k\to\infty}\bm{F}_n^{(k)}=\bm{F}_n^{(1)}\bm{\mathfrak{B}}_n^{-1}.
\end{equation}
Moreover, $\mathbb{B}_{n}^{(\infty)}$ preserves linear functions.
\end{theorem}
\begin{proof} Since all the eigenvalues of matrix
$\bm{\mathfrak{B}}_n$ are in $(0, 1]$ and exactly two of them are
ones, all the eigenvalues of matrix $I_{n+1}-\bm{\mathfrak{B}}_n$
are in $[0, 1)$ and exactly two of them are zeros. Thus
$\lim_{k\to\infty}(I_{n+1}-\bm{\mathfrak{B}}_n)^k =\bm{O}$, the zero
matrix. Because $\mathbb{B}_{n}^{(k)}$ preserves linear functions
for any positive integer $k$, so does $\mathbb{B}_{n}^{(\infty)}$.
This can also be proved by the following facts that
$$\bm{F}_n^{(1)}\bm{\mathfrak{B}}_n=\bm{F}_n^{(1)}\quad \mbox{if and only if}\quad \bm{F}_n^{(1)}\bm{\mathfrak{B}}_n^{-1}=\bm{F}_n^{(1)}$$
and that $\bm{F}_n^{(1)}\bm{\mathfrak{B}}_n=\bm{F}_n^{(1)}$ is true
provided that $f$ is linear.

\end{proof}
Numerical examples (see \S \ref{sect: numerical example}) show that
the maximum absolute approximation error seems to be minimized by
``optimal'' \bernstein polynomial approximation
$\mathbb{B}_{n}^{(\infty)}f(t)$ if $f$ is infinitely differentiable.
For nonsmooth functions such as $f(t)=|t-0.5|$ and fixed $n$, it
seems that the maximum absolute approximation error is minimized by
the iterated \bernstein polynomial approximation
$\mathbb{B}_{n}^{(k)}f(t)$ for some $k$.

 The next theorem shows that if $k>1$ then
$\mathbb{B}_{n}^{(k)}f(t)$ is indeed a better polynomial
approximation of $f$ than the classical \bernstein polynomial.
\begin{theorem}\label{thm: rate of convergence of the k-th Bernstein poly approx}
Suppose that $f\in C_{d_{kr}}[0,1]$, $d_{kr}=2(k-1)+r$ and $r=0, 1$. 
Then
\begin{equation}\label{eq: rate of convergence of the k-th Bernstein poly approx}
|\Err\{\mathbb{B}_{n}^{(k)}f(t)\}| =|\mathbb{B}_{n}^{(k)}f(t)-f(t)|\le C_{kr}''n^{-\frac{d_{kr}}{2}}\omega_{d_{kr}}(n^{-1/2}),
\end{equation}
where $C_{kr}''$ is a constant depending on  $r$ and $k$ only.
\end{theorem}
\begin{proof}
This result follows easily from Theorems \ref{rate of conergence of berinstein polynomial approx}
and \ref{thm of Voronovskaya}.
\end{proof}
\begin{rem}
From this theorem with $k=2$ and $r=0$, we see that
   if $f$ has continuous second derivative then the rate of convergence of the second \bernstein polynomial approximation $\mathbb{B}_{n}^{(2)}f$ is at least $o(n^{-1})$.
\end{rem}
\begin{rem}
From Theorem \ref{thm: rate of convergence of the k-th Bernstein poly approx} with $k=2$ we see that if $f$ has continuous fourth derivative, then the rate of convergence of $\mathbb{B}_{n}^{(2)}f$
can be as fast as $\mathcal{O}(n^{-2})$. This seems the fastest rate that $\mathbb{B}_{n}^{(2)}f$ can reach even if $f$ has continuous fifth or higher derivatives.
\end{rem}
\begin{rem}
 It can also be  proved that if $f$ has continuous $(2k)$th derivative, then the rate of
  convergence of $\mathbb{B}_{n}^{(k)}f$
can be as fast as $\mathcal{O}(n^{-k})$. Although these improvements upon $\mathbb{B}_{n}f(t)$
are still not as good as those stated in Theorem \ref{thm of Dunham Jackson}, they are good enough for application in computer graphics and statistics.
\end{rem}
\begin{rem}
 It is a very interesting project to investigate the relationship between $C_{kr}''$ and $k$, and  the rate
 of convergence of $\mathbb{B}_{n}^{(\infty)}f$ which is conjectured
 to be exponential.
\end{rem}
\section{The Derivatives and Integrals of $\mathbb{B}_{n}^{(k)}f(t)$ and Applications}
\subsection{The Derivatives of $\mathbb{B}_{n}^{(k)}f(t)$}
\begin{theorem} \label{thm: derivatives of iBernstein poly}
 For any positive integers $k$ and $r\le n$,
 \begin{equation}\label{eq: r-th derivatives of iBernstein poly}
    \frac{d^r}{dt^r}\mathbb{B}_{n}^{(k)}f(t)
   =\frac{n!}{(n-r)!}\sum_{i=0}^{n-r} \sum_{j=1}^{k}(-1)^{j-1} {k\choose j} \Delta^r(\mathbb{B}_{n}^{j-1}f)\big(\sfrac({i},{n})\big) B_{n-r,i}(t),
\end{equation}
where $\Delta^r$ is the $r$th forward difference operator with increment $h=1/n$, $\Delta f(t)=f(t+h)-f(t)$,
$$\Delta^{r}f(t)=\sum_{i=0}^r {r\choose i} (-1)^i f\big(t+\sfrac(r-i,h)\big).$$
\end{theorem}
\begin{proof}
If $k=1$, it is well known that for any function $f$
\begin{equation}\label{eq: derivatives of iBernstein poly k=1}
\frac{d}{dt}\mathbb{B}_{n}^{(1)}f(t)
=\frac{d}{dt}\mathbb{B}_{n}f(t)
=n\sum_{i=0}^{n-1} \Delta f\big(\sfrac({i},{n})\big) B_{n-1,i}(t).
\end{equation}
Assume that (\ref{eq: r-th derivatives of iBernstein poly}) with $r=1$ is true for the $k$th iterated \bernstein polynomial of any function $f$. By  (\ref{eq: inductive formula of iterated bernstein polynomial}) we have
\begin{align}\nonumber\label{eq: (k+1)st derivative of iBernstein-1}
        \frac{d}{dt}\mathbb{B}_{n}^{(k+1)}f(t)
   &=\frac{d}{dt}\left[\mathbb{B}_{n}^{(k)}f(t)-\mathbb{B}_{n}\{\mathbb{B}_{n}^{(k)}f(t)-f(t)\}\right]\\
   &=\frac{d}{dt} \mathbb{B}_{n}^{(k)}f(t)+\frac{d}{dt}\mathbb{B}_{n}f(t)-\frac{d}{dt}\mathbb{B}_{n}\{\mathbb{B}_{n}^{(k)}\}f(t).
\end{align}
It follows from (\ref{eq: derivatives of iBernstein poly k=1})
and (\ref{k-th Bernstein poly of order n}) that
\begin{align}\nonumber
\frac{d}{dt}\mathbb{B}_{n}\{\mathbb{B}_{n}^{(k)}\}f(t)
&=n\sum_{i=0}^{n-1}\Delta \mathbb{B}_{n}^{(k)}f\big(\sfrac({i},{n})\big)B_{n-1,i}(t)\\
\nonumber\label{eq: (k+1)st derivative of iBernstein-2}
&=n\sum_{i=0}^{n-1}\sum_{j=0}^{n}
f\big(\sfrac({j},{n})\big)\sum_{\ell=1}^k{k\choose \ell}(-1)^{\ell-1}\Delta\mathbb{B}_{n}^{\ell-1}B_{nj}\big(\sfrac({i},{n})\big) B_{n-1,i}(t)
\\
&=n\sum_{i=0}^{n-1} \sum_{\ell=1}^k{k\choose \ell}(-1)^{\ell-1}\Delta\mathbb{B}_{n}^{\ell}f\big(\sfrac({i},{n})\big)B_{n-1,i}(t).
\end{align}
Combining  (\ref{eq: derivatives of iBernstein poly k=1}),  (\ref{k-th Bernstein poly of order n}), (\ref{eq: (k+1)st derivative of iBernstein-1}), and (\ref{eq: (k+1)st derivative of iBernstein-2}) we arrive at
\begin{equation}\label{eq: (k+1)st derivative of iBernstein}
    \frac{d}{dt}\mathbb{B}_{n}^{(k+1)}f(t)=n\sum_{i=0}^{n-1} \sum_{j=1}^{k+1}(-1)^{j-1} {k+1\choose j} \Delta\mathbb{B}_{n}^{j-1}f\big(\sfrac({i},{n})\big) B_{n-1,i}(t).
\end{equation}
The proof of (\ref{eq: r-th derivatives of iBernstein poly}) with $r=1$ and $k\ge 1$ is complete by induction. Similarly
(\ref{eq: r-th derivatives of iBernstein poly}) with $r\ge1$ and $k\ge 1$ can be proved using induction.
\end{proof}
It is not hard to prove by adopting the method of \cite{Lorentz-1986-book-bernstein-poly} that
\begin{theorem}\label{thm: convergence of derivatives of iBernstein polynomials}
    (i) If $f$ has continuous $r$th derivative $f^{(r)}$ on $[0, 1]$, then for each fixed $k$, as $n\to\infty$,  $\frac{d^r}{dt^r}\mathbb{B}_{n}^{(k)}f(t)$ converge to $f^{(r)}(t)$ uniformly on $[0, 1]$.
\\
(ii) If $f$ in bounded on $[0, 1]$ and its $r$th derivative $f^{(r)}(t)$ exists at $t\in [0, 1]$, then for each fixed $k$, as $n\to\infty$,  $\frac{d^r}{dt^r}\mathbb{B}_{n}^{(k)}f(t)$ converge to $f^{(r)}(t)$.
\end{theorem}
Numerical examples show that the larger the $r$ is, the slower the above convergence is.

For any positive integers $k$, the second  derivative of the
iterated \bernstein polynomial $\mathbb{B}^{(k)}f$ is
 \begin{equation}\label{eq: 2nd derivatives of iBernstein poly}
    \frac{d^2}{dt^2}\mathbb{B}_{n}^{(k)}f(t)
   =n(n-1)\sum_{i=0}^{n-2} \sum_{j=1}^{k}(-1)^{j-1} {k\choose j}
\Delta^2(\mathbb{B}_{n}^{j-1}f)\big(\sfrac({i},{n})\big)  B_{n-2,i}(t).
\end{equation}
It is well known that if $f$ is convex on $[0, 1]$, then $\frac{d^2}{dt^2}\mathbb{B}_{n}^{(1)}f(t)\ge 0$ and thus $\mathbb{B}_{n}^{(1)}f(t)$ is also convex and  $\mathbb{B}_{n}^{(1)}f(t)\ge f(t)$ on $[0, 1]$.
 So the classical \bernstein polynomials preserve the convexity of the original function and
 has nonnegative errors. However examples of \S \ref{sect: numerical example} show that when $k\ge 2$ the iterated \bernstein polynomial $\mathbb{B}^{(k)}f$ does not preserve the convexity of the original function unconditionally.
The iterated \bernstein polynomials still preserve the monotonicity
of $f$ if it is not too ``flat'' anywhere.
\begin{theorem}\label{thm: preservation of monotonicity}
If $f$ is strictly increasing (decreasing) on $[0,1]$, for any $k\ge
1$, $\mathbb{B}_n^{(k)}f(t)$ is also strictly increasing
(decreasing) on $[0,1]$.
\end{theorem}
\begin{proof}
The theorem is true for $k=1$ even if $f$ is increasing
(decreasing), but not strictly, on $[0,1]$. It suffices to prove the
theorem when $f$ is strictly increasing on $[0,1]$.  Assume that the
theorem is true for some $k\ge1$. Since $f$ is strictly increasing
on $[0,1]$, $\mathbb{B}_n^k f(t)$ are also strictly increasing
 on $[0,1]$ for all $k\ge1$.
\end{proof}
\begin{rem}
If $k=1$, the condition of strict monotonicity is not necessary.
However, if $k>1$, the condition of strict monotonicity can be
relaxed. For example, $f(x)=x$, if $0\le x< 1/3$, $=1/3$, if $1/3\le
x <2/3$, and $=x-1/3$, if $2/3\le x\le 1$. It can be shown that
$\frac{d}{dt}\mathbb{B}_{n}^{(2)}f(x)<0$ for $x$ in a neighborhood
of $x=1/2$.
\end{rem}
 \subsection{The Integrals of $\mathbb{B}_{n}^{(k)}f(t)$}
The following theorem is very useful for implementing the iterative algorithm in computer languages.
\begin{theorem} \label{thm: integrals of iBernstein poly}  Suppose $f$ is continuous on $[0, 1]$.
For $1\le k\le \infty$ and $x\in [0, 1]$, we have
\begin{equation}\label{eq: integral of iBernstein polynomials from 0 to x}
\int_0^x\mathbb{B}_{n}^{(k)}f(t)dt  = \sum_{i=0}^{n}
f_{ni}^{(k)}S_{ni}(x)=\bm{F}_n^{(k)} \bm{S}_n(x),
\end{equation}
where $\bm{S}_n(x)=\{S_{n0}(x),\ldots,S_{nn}(x)\}^\tr$ and
$$S_{ni}(x)=\int_0^x B_{ni}(t)dt=\sfrac(1,n+1)\int_0^x \beta_{ni}(t)dt, \quad S_{ni}(1)=\sfrac(1,n+1).$$
\end{theorem}
\begin{corollary} \label{corollary: integrals approximation based on iBernstein poly}
If $g$ is continuous on $[a, b]$, $a<b$, then for $1\le k\le
\infty$,
\begin{equation}\label{eq: integrals approximation based on iBernstein polynomials from a to b}
\int_a^b g(t)dt\approx  \frac{1}{n+1}\sum_{i=0}^{n}f_{ni}^{(k)}=\frac{1}{n+1}\bm{F}_n^{(k)} \bm{1}_{n+1},
\end{equation}
where $\bm{F}_n^{(k)}$ is calculated based on
 $f(t)=\sfrac(1,b-a) g[a+(b-a)t] $.
\end{corollary}
\begin{rem}
Note that numerical integration (\ref{eq: integrals approximation
based on iBernstein polynomials from a to b}) does not involve any
integrals. It contains only algebraic calculations. See Example
\ref{example 9} of \S \ref{sect: numerical example} for some
numerical examples.
\end{rem}
\begin{proof} The theorem follows immediately from Theorems \ref{thm: formula for matrix  Fn}
and \ref{thm: formula for matrix  Fn when k=infty}.
\end{proof} The
following theorem follows immediately from Theorems \ref{thm: rate
of convergence of the k-th Bernstein poly approx} and \ref{thm:
integrals of iBernstein poly}.
\begin{theorem} Under the condition of Theorem \ref{thm: rate of convergence of the k-th Bernstein poly approx}, for
any $x\in [0, 1]$
\begin{equation}\label{ineq: error estimate of integral approximation}
    \left|\int_0^x\mathbb{B}_{n}^{(k)}f(t)dt -\int_0^x f(t)dt\right|
\le C_{kr}''n^{-\frac{d_{kr}}{2}}\omega_{d_{kr}}(n^{-1/2}),
\end{equation}
where $C_{kr}''$ is a constant depending on  $r$ and $k$ only.
\end{theorem}
\section{Iterated Szasz  Approximation and Iterated $q$-\bernstein Polynomial}
The idea used to construct the iterated \bernstein polynomial approximation is simple and very effective. The same idea seems also applicable to  other operators or approximations such as the Szasz  operator \citep{Otto-Szasz-1950} [or the Szasz-Mirakyan (Mirakja) operator]  and the $q$-\bernstein polynomial with $q>1$. We will give some numerical examples in \S \ref{sect: numerical example} and the analogues of results of Section \ref{sect: iterated Bernstein poly and rate of convergence} could be be obtained by using the analogue results about the rate of convergence of the  Szasz-Mirakyan approximation \citep{Totik-AJM-1994}.
We hope these would inspire more investigations with rigorous mathematics.
\subsection{Iterated Szasz  Approximation}
The so-called Szasz-Mirakyan approximation is defined as
\begin{equation}\label{Szasz-Mirakyan operator}
    \mathbb{S}_nf(x)=\sum_{i=0}^\infty f\big(\sfrac({i},{n})\big) P_{ni}(x),\quad x\in [0, \infty),
\end{equation}
where $f$ is defined on $[0, \infty)$ and $P_{ni}(x)=e^{-nx} {(nx)^i}/{i!}$.
Note that, for $x>0$, $P_{ni}(x)$ is the probability that $V_n(x)=i$  where $V_n(x)$ is the Poisson random variable  with mean $nx$. Since the   binomial probability $B_{ni}(t)$ can be approximated by  $P_{ni}(t)$ for large $n$, the
 Szasz-Mirakyan approximation can be viewed as an extension of the \bernstein polynomial approximation. The error of $\mathbb{S}_nf$ as an approximation of $f$ is
\begin{equation}\label{Error of Szasz-Mirakyan operator}
    \Err(\mathbb{S}_nf)(x)=\mathbb{S}_nf(x)-f(x)=\sum_{i=0}^\infty f\big(\sfrac({i},{n})\big) P_{ni}(x)-f(x),\quad x\in [0, \infty).
\end{equation}
Applying the Szasz-Mirakyan operator to $\Err(\mathbb{S}_nf)(x)$, we have
\begin{equation}\label{Approx of Error of Szasz-Mirakyan operator}
    \mathbb{S}_n\{\Err(\mathbb{S}_nf)\}(x)=\mathbb{S}_n^2f(x)-\mathbb{S}_nf(x)=\sum_{i=0}^\infty f\big(\sfrac({i},{n})\big) \mathbb{S}_n P_{ni}(x)-\mathbb{S}_nf(x),\quad x\in [0, \infty).
\end{equation}
So we can define the second Szasz-Mirakyan approximation as
\begin{equation}\label{the 2nd Szasz-Mirakyan operator}
    \mathbb{S}_n^{(2)}f(x)=\mathbb{S}_nf(x)-\mathbb{S}_n\{\Err(\mathbb{S}_nf)\}(x),\quad x\in [0, \infty).
\end{equation}
\begin{theorem}
\begin{align}
\label{The k-th-Szasz-Mirakyan approximation}
    \mathbb{S}_{n}^{(k)}f(x)
   & =\sum_{i=0}^{\infty}
f\big(\sfrac({i},{n})\big)\sum_{j=1}^k{k\choose j}(-1)^{j-1}\mathbb{S}_{n}^{j-1}P_{ni}(x) ,\quad k\ge 1,\quad x\in [0, \infty).
\end{align}
Clearly,
for every $k\ge 1$, $\mathbb{S}_{n}^{(k)}$ preserves linear functions and therefore
\begin{equation}\label{error of k-th Szasz-Mirakyan approximation}
\Err\{\mathbb{S}_{n}^{(k)}f(x)\}=\sum_{i=0}^{\infty}
\left\{f\big(\sfrac({i},{n})\big)-f(x)\right\}\sum_{j=1}^k{k\choose j}(-1)^{j-1}\mathbb{S}_{n}^{j-1}P_{ni}(t)  ,\quad k\ge 1;\, x\in [0, \infty).
\end{equation}
\end{theorem}
Figure \ref{fig: error-Iterated-Szasz-Approx-dchisq-4} gives an example of the iterated Szasz approximations.
\subsection{Iterated $q$-\bernstein Polynomial}\label{sect: q-Bernstein polynomials}
Let $x$ be a real number. For any $q>0$, define the $q$-number
$$[x]_q=\left\{
          \begin{array}{ll}
            \frac{1-q^x}{1-q}, & \hbox{if $q\ne 1$;} \\
            x, & \hbox{if $q=1$.}
          \end{array}
        \right.$$
If $x$ is integer, then $[x]_q$ is called a $q$-integer. For $q\ne 1$, the $q$-binomial coefficient (Gaussian binomial) is defined by
$${n\choose r}_{\!q}=
\left\{
  \begin{array}{ll}
    1, & \hbox{$r=0$;} \\
    \frac{(1-q^n)(1-q^{n-1})\cdots (1-q^{n-r+1})}
{(1-q^r)(1-q^{r-1})\cdots (1-q)}, & \hbox{$1\le r\le n$;} \\
    0, & \hbox{$r>n$.}
  \end{array}
\right.
 $$
So
$${n\choose r}_{\!q}=\prod_{i=0}^{r-1}\left[\frac{n-i}{r-i}\right]_{q^{r-i}},\quad 0\le r\le n, \quad q>0,$$
where empty product is defined to be 1. Thus the ordinary binomial coefficient ${n\choose r}$ is the special case when $q=1$. G. M. Phillips \cite{Phillips-G-M-1997-q-Bernstein} introduced the $q$-\bernstein polynomial of order $n$ for any continuous function $f(t)$ on the interval $[0, 1]$
$$ \mathbb{Q}_{nq}f(t)=\sum_{i=0}^n f\left(t_{i}^{(q)}\right)Q_{ni}(t),\quad n=1,2,\ldots,$$
where
$$t_{i}^{(q)}=\frac{[i]_q}{[n]_q},\quad Q_{ni}(t)={n\choose i}_{\!q}
t^i\prod_{j=1}^{n-i}(1-tq^{j-1}),\quad i=0, 1, \ldots,n.$$
Clearly, $\mathbb{B}_nf(t)=\mathbb{Q}_{n1}f(t)$ which is the classical \bernstein polynomial of order $n$. It has been proved that if $0<q<1$ then $\mathbb{Q}_{nq}f(t)$ does not approximate $f$  and that if $q>1$  and $f(z)$ is analytic complex function on disk $\{z\,:\, |z|<r\}$, $r>q$, then $\mathbb{Q}_{nq}f(t)$ has better rate of convergence, $\mathcal{O}(q^{-n})$, than the best rate of convergence, $\mathcal{O}(n^{-1})$, of  $\mathbb{B}_nf(t)$ \citep[see][for example]{Ostrovska-2003-JAT,Wang-Wu-2008-JMAA-Saturate-Conv-q-Bernstein}.

Note that if $q>1$ then points $t_{i}^{(q)}= {[i]_q}/{[n]_q}$ are no longer uniform partition points of the interval $[0, 1]$. For fixed $n$, $\lim_{q\to\infty}t_{i}^{(q)}=0$, $i<n$.  So all $t_i$ except $t_{n}^{(q)}=1$ are attracted toward 0 as $q$ getting large.
However, interestingly, the larger the $q$ is in a certain range, the closer the $q$-\bernstein polynomial approximation $\mathbb{Q}_{nq}f(t)$ to $f(t)$. For a given $n$, if $q$ is too large, the $q$-\bernstein polynomial approximation $\mathbb{Q}_{nq}f(t)$ becomes worse in the neighborhood of the right end-point.

Similarly we have the iterated $q$-\bernstein polynomials
\begin{align}
\label{The k-th-q-Bernstein poly}
    \mathbb{Q}_{nq}^{(k)}f(t)
   & =\sum_{i=0}^{\infty}
f\big(t_{i}^{(q)}\big)\sum_{j=1}^k{k\choose j}(-1)^{j-1}\mathbb{Q}_{nq}^{j-1} Q_{ni}(t) ,\quad k\ge 1,\quad t\in [0, 1].
\end{align}
See Figure \ref{fig: error-Iterated-qBernstein-Poly-sin} for an example of the iterated $q$-\bernstein polynomials. Comparing
Figures \ref{fig: error-Iterated-Bernstein-Poly-sin} and \ref{fig: error-Iterated-qBernstein-Poly-sin} we see that increasing $q$ from 1 to 1.1 does improve the approximation on $[0, 1]$ except at points in the neighborhood of the right end-point. The approximation near the right end-point could be worsen by applying the iterated $q$-\bernstein polynomials. The improvement can be achieved by the iterated \bernstein polynomials without messing up the right boundary.

\section{Numerical Examples}\label{sect: numerical example}
In this  section some numerical examples are given with the hope of more
 investigations on the proposed methods with rigorous mathematics.
\begin{example} \label{example 1}
Figure  \ref{fig: error-Iterated-Bernstein-Poly-sin} shows  the
first three iterated \bernstein polynomials of $f(t)=\sin(2\pi t)$
and the errors where $n=30$. The ``optimal'' \bernstein polynomial
approximation is also plotted which seems to have almost no error.
\end{example}
\begin{example} \label{example 2}
 Figure  \ref{fig: error-Iterated-Bernstein-Poly-not-twice-differentiable-function} shows  the first three iterated \bernstein polynomials of $f(t)=\mathrm{sign}(t-0.5)(t-0.5)^2$ (a differentiable but not twice differentiable function) and the errors where $n=30$.
The ``optimal'' \bernstein polynomial approximation is not plotted
which becomes very bad near the two endpoints.
\end{example}
\begin{example} \label{example 3}
Figure  \ref{fig:
error-Iterated-Bernstein-Poly-not-differentiable-function} shows the
first three iterated \bernstein polynomials of $f(t)=|t-0.5|$ and
the errors where $n=30$. The ``optimal'' \bernstein polynomial
approximation is not plotted which becomes very bad near the two
endpoints.
\end{example}
\begin{example} \label{example 4}
 Figure  \ref{fig: error-Iterated-Szasz-Approx-dchisq-4} shows  the first three iterated Szasz  approximation of $f(x)=0.25 x e^{-x/2}$, $x\ge 0$, and the errors where $n=10$.

\end{example}
\begin{example} \label{example 5}
 Figure  \ref{fig: error-Iterated-qBernstein-Poly-sin} shows  the first three iterated $q$-\bernstein  polynomials of $f(x)=\sin(\pi x)$  and the errors where $n=30$, $q=1.1$. The performance of the approximation near $t=1$ is very sensitive to $q$.
\end{example}
\begin{example} \label{example 6}
 Figure  \ref{fig: Derivative-Iterated-Bernstein-Poly-convex-not-differentiable-function} shows  the first three iterated \bernstein  polynomials of the following function $f(t)=|t-0.5|$  and their derivatives where $n=30$.
\end{example}
\begin{example} \label{example 7}
 Figure  \ref{fig: Derivative-Iterated-Bernstein-Poly-convex-not-twice-differentiable-function} shows  the first three iterated \bernstein  polynomials of the following function $f(t)$  and their derivatives where $n=30$,
$$f(t)=
\left\{
          \begin{array}{ll}
            t(t-1), & \hbox{$0\le t<0.5$;}\\
            -\sfrac(1,4)+\sfrac(2,3)(t-0.5)^{3/2}, & \hbox{$0.5\le t\le 1$.}
          \end{array}
        \right.
$$
This a convex function which has continuous first  derivative but does not have a continuous
second derivative.
\end{example}
\begin{example} \label{example 8}
 Denote $t_\delta=\sfrac(2,3)-\delta$ where $\delta$ is a small positive number.
$$f(t)=\left\{
  \begin{array}{ll}
    f_0(t), & \hbox{$0\le t\le t_\delta$;} \\
    p_k(t), & \hbox{$t_\delta< t\le 1$,}
  \end{array}
\right.$$
where $f_0(t)=v-\sqrt{r^2-(t-u)^2}$ is portion of a circle with radius $r$ (a larger positive number) and centered at $(u, v)$, $u, v>0$, $p_k(t)$ is a polynomial of degree $k=3$,
$$p_k(t)=\sum_{i=0}^ka_{ki}t^i=a_{kk}t^k+a_{k,k-1}t^{k-1}+\cdots+a_{k1}t+a_{k0}.$$
If we choose $$v=\frac{-30t_\delta\pm \sqrt{900t_\delta^2-40(25t_\delta^2-r^2)}}{20},\quad u=\sqrt{r^2-v^2}$$
then $f(0)=f_0(0)=0$, $f(t_\delta)=f_0(t_\delta)=-3t_\delta.$ We also have
$$f_0'(t)=\frac{t-u}{\sqrt{r^2-(t-u)^2}},\quad f_0''(t)=\frac{r^2}{\{r^2-(t-u)^2\}^{3/2}}.$$
Choose the coefficients of $p_k$ so that
$f(1)=\sum_{i=0}^ka_{ki}=0$ and
the $j$th ($j=0,1,\ldots,k-1$) derivative at $t_\delta$ satisfy
$$f^{(j)}(t_\delta)=\sum_{i=j}^k\sfrac(i!,{(i-j)!})a_{ki}t_\delta^{i-j}
=f_0^{(j)}(t_\delta^{1-j}).$$
If $r$ is large enough, say $r=70$, $\delta=0.05$, then $f(t)$ is strictly convex and has continuous positive second derivative $f''$, but $\mathbb{B}_n^{(2)}f$ is still not convex because its second derivative is negative at some points near $t=0.4$ (see  Figure \ref{fig: Example-5-strictly-convex-continuous-positive-second-derivative}).
\end{example}

\begin{example}\label{example 9}
In the following Tables \ref{tbl: example 9 integrals n=5} and
\ref{tbl: example 9 integrals n=10} we summarize some the results of
numerical integrals on $[0, 1]$ using our proposed method given in
Corollary \ref{corollary: integrals approximation based on
iBernstein poly} for functions $f(x)=\pi\sin (\pi x)$, $f(x)=e^x$,
and $f(x)=\varphi(x)=(1/\sqrt{2\pi})\exp(-x^2/2)$.
\begin{table}
  \centering
  \caption{Some results of numerical
integrals ($n=5$)}\label{tbl: example 9 integrals n=5}
\begin{tabular}{|c|c|c|c|c|}
  \hline
   & \multicolumn{3}{c|}{$k$}&\\\hline
    & 1 & 5 & $\infty$ & Exact value\\
  \hline
  $\int_0^1\pi\sin (\pi x)dx$ & 1.611471&2.005416 & 1.999203&2 \\
  \hline
  $\int_0^1e^xdx$ & 1.746528&1.718369 & 1.718282&1.718282 \\
  \hline
  $\int_0^1\varphi(x)dx$ & 0.3371903&0.3413510 & 0.3413443&0.3413447\\
  \hline
\end{tabular}

\end{table}
\begin{table}
  \centering
  \caption{Some results of numerical
integrals ($n=10$)}\label{tbl: example 9 integrals n=10}
\begin{tabular}{|c|c|c|c|c|}
  \hline
 & \multicolumn{3}{c|}{$k$}&\\\hline
   & 1 & 5 & $\infty$ & Exact value \\
  \hline
  $\int_0^1\pi\sin (\pi x)dx$ & 1.803203&2.000146 & 2.000000&2 \\
  \hline
  $\int_0^1e^xdx$ & 1.732389&1.718285 &1.718282&1.718282 \\
  \hline
  $\int_0^1\varphi(x)dx$ & 0.3392624&0.341345 & 0.3413447&0.3413447\\
  \hline
\end{tabular}

\end{table}

\end{example}

From these examples and the figures we see that the error is reduced
significantly by using the iterated \bernstein polynomial
approximation without increasing the degree of the polynomial. For
non-smooth function, the maximum error is reduced more than 50\% by
the third \bernstein polynomial. It is also seen from Figure
\ref{fig: error-Iterated-Bernstein-Poly-not-differentiable-function}
that unlike the classical \bernstein polynomial approximation the
iterated \bernstein polynomial approximation $\mathbb{B}_{n}^{(k)}f$
seems not to preserve the convexity of $f$ for $k>1$ in this case
when $f$ is not smooth. So it is necessary for
$\mathbb{B}_{n}^{(k)}f$ to preserve the convexity of $f$   that $f$
is smooth and $f''$ is not too close to zero. For applications in
numerical integrals and computer graphics, sometimes it is even much
more expensive to evaluate the function $f$ than the simple
algebraic calculations. So it is significant to apply the iterated
or the ``optimal'', if $f$ is infinitely differentiable, \bernstein
polynomial approximation.

\def\polhk#1{\setbox0=\hbox{#1}{\ooalign{\hidewidth
  \lower1.5ex\hbox{`}\hidewidth\crcr\unhbox0}}} \def\cprime{$'$}
  \def\lfhook#1{\setbox0=\hbox{#1}{\ooalign{\hidewidth
  \lower1.5ex\hbox{'}\hidewidth\crcr\unhbox0}}} \def\cprime{$'$}

\begin{figure}
\begin{center}
  \includegraphics[width=6.0in]{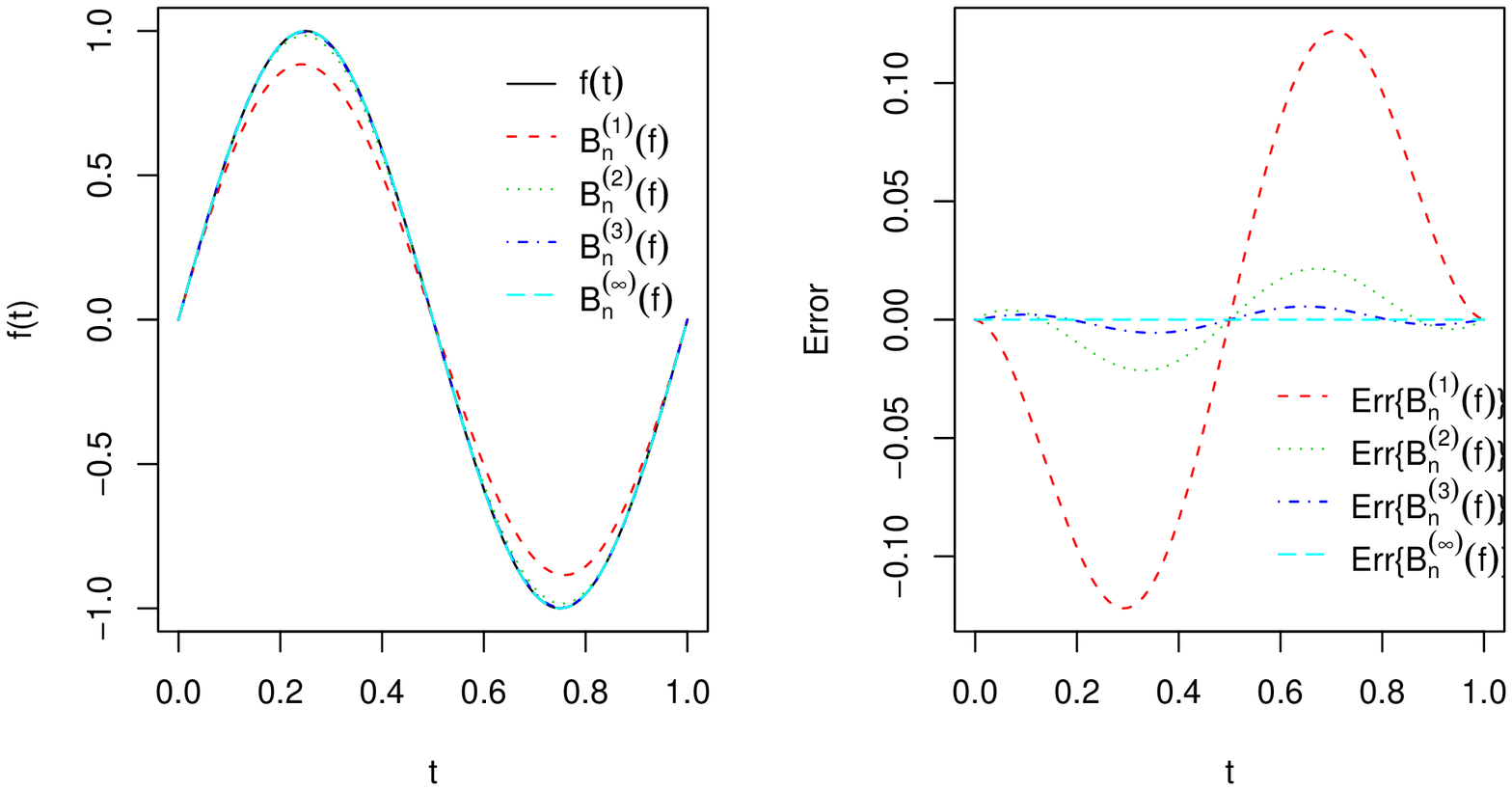}
  \caption{The iterated Bernstein polynomials and errors when $f(t)=\sin(2\pi t)$.
  The error is minimized by $\mathbb{B}_{n}^{(\infty)}f$.}\label{fig: error-Iterated-Bernstein-Poly-sin}
\end{center}
\end{figure}

\begin{figure}
\begin{center}
  \includegraphics[width=6.0in]{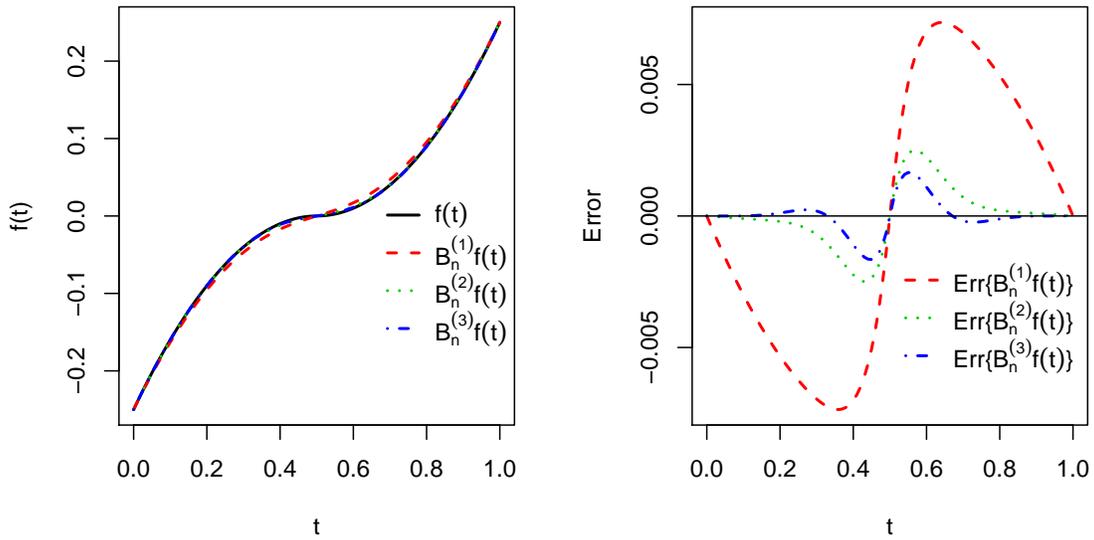}
  \caption{The iterated Bernstein polynomials and errors when $f(t)=\mathrm{sign}(t-0.5)(t-0.5)^2$ which is differentiable on $[0, 1]$ but not twice differentiable at $t=0.5$.}\label{fig: error-Iterated-Bernstein-Poly-not-twice-differentiable-function}
\end{center}
\end{figure}
\begin{figure}
\begin{center}
  \includegraphics[width=6.0in]{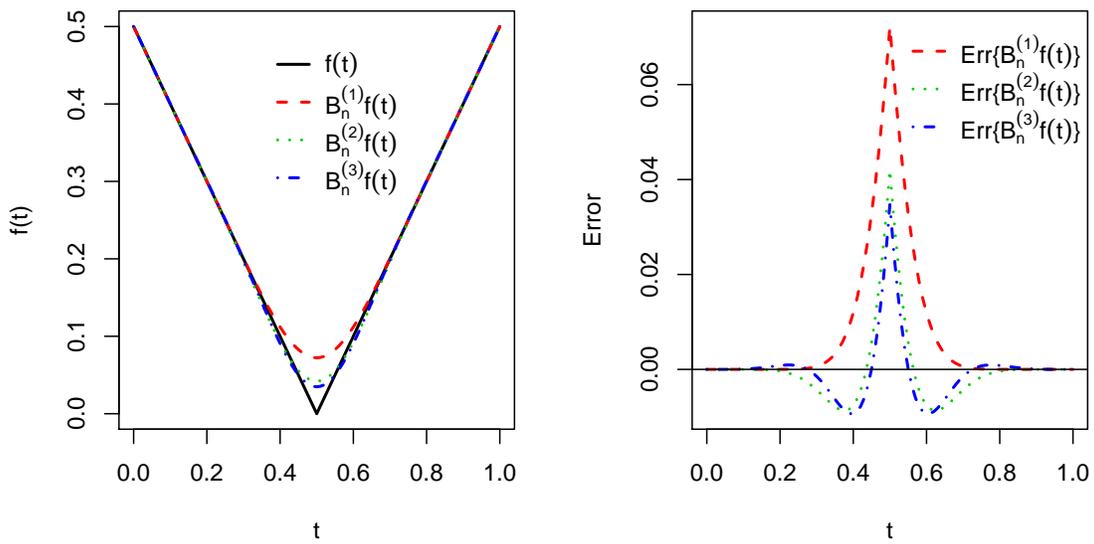}
  \caption{The iterated Bernstein polynomials and errors when $f(t)=|t-0.5|$ which is not differentiable at $t=0.5$.}\label{fig: error-Iterated-Bernstein-Poly-not-differentiable-function}
\end{center}
\end{figure}

\begin{figure}
\begin{center}
  \includegraphics[width=6.0in]{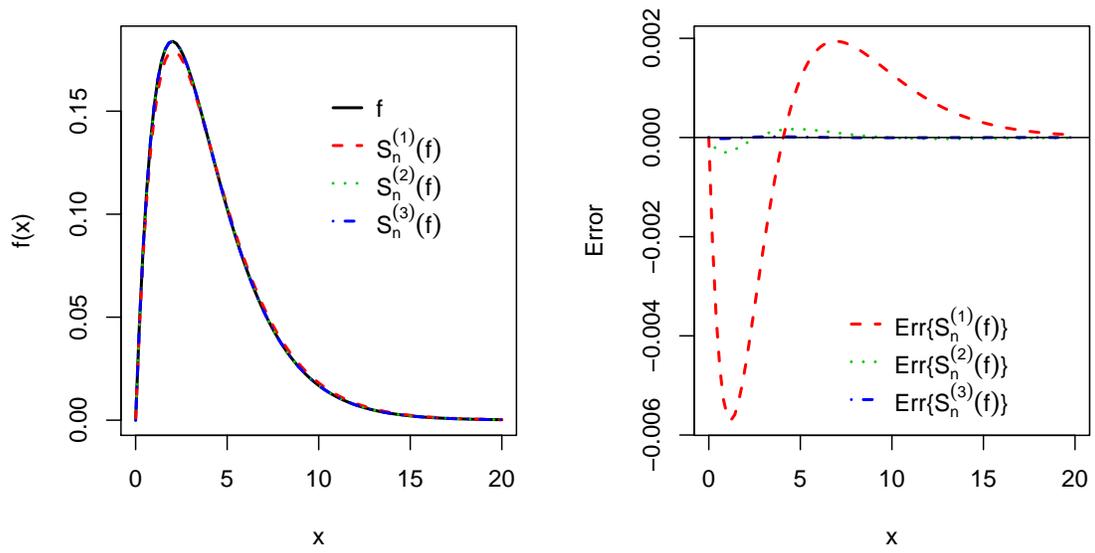}
  \caption{The iterated Szasz approximations and errors when $f(x)=0.25 x e^{-x/2}$, $x\ge 0$  with $n=10$.}\label{fig: error-Iterated-Szasz-Approx-dchisq-4}
\end{center}
\end{figure}

\begin{figure}
\begin{center}
  \includegraphics[width=6.0in]{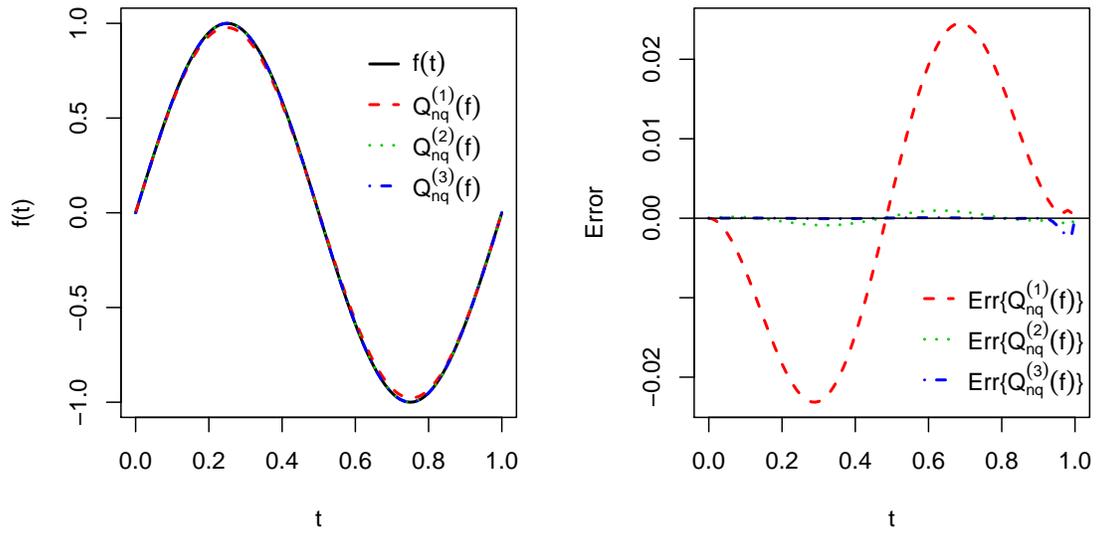}
  \caption{The iterated $q$-Bernstein polynomials and errors when $f(t)=\sin(2\pi t)$ with $n=30$, $q=1.1$.}\label{fig: error-Iterated-qBernstein-Poly-sin}
\end{center}
\end{figure}
\begin{figure}
\begin{center}
  \includegraphics[width=6.0in]{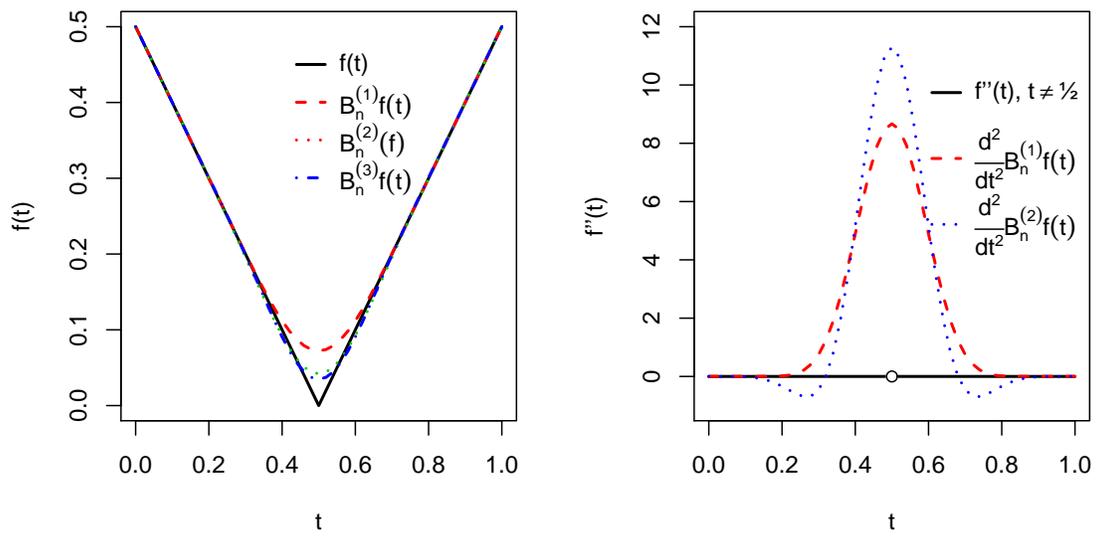}
  \caption{The iterated Bernstein polynomials and their derivatives when $f(t)=|t-0.5|$ which is convex but not differentiable at $t=0.5$.}\label{fig: Derivative-Iterated-Bernstein-Poly-convex-not-differentiable-function}
\end{center}
\end{figure}

\begin{figure}
\begin{center}
  \includegraphics[width=6.0in]{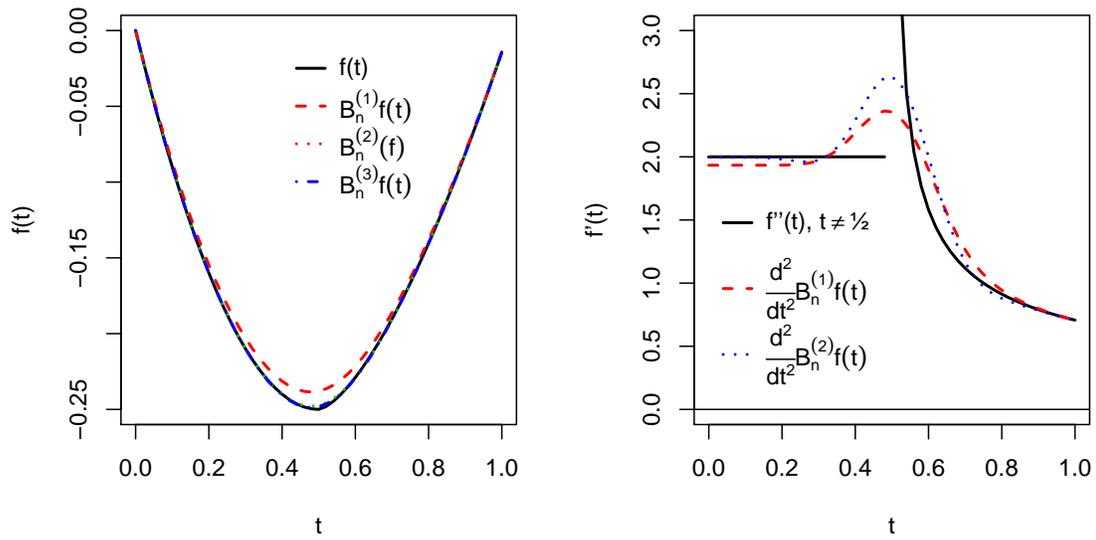}
  \caption{The iterated Bernstein polynomials of $f(t)$ as in Example \ref{example 7} and their derivatives  where $f(t)$  is convex,  differentiable on $[0, 1]$ but not twice differentiable at $t=0.5$.}\label{fig: Derivative-Iterated-Bernstein-Poly-convex-not-twice-differentiable-function}
\end{center}
\end{figure}
\begin{figure}
\begin{center}
  \includegraphics[width=6.0in]{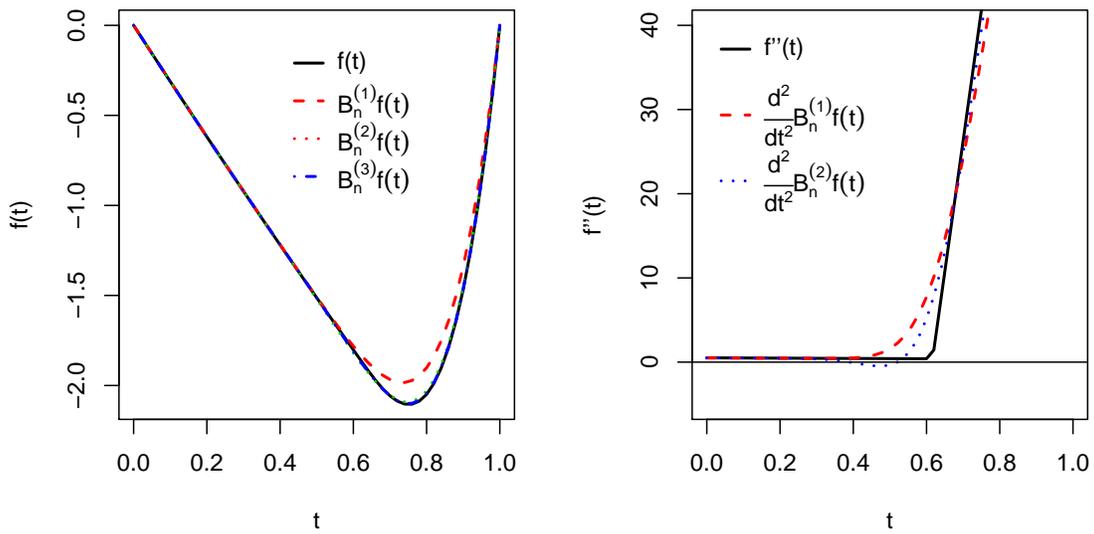}
  \caption{The iterated Bernstein polynomials of $f$ as in Example \ref{example 8} and their derivatives. The function $f$ is strictly convex but $\mathbb{B}_n^{(2)}f$ is not convex.}\label{fig: Example-5-strictly-convex-continuous-positive-second-derivative}
\end{center}
\end{figure}

\end{document}